\documentclass [11pt] {article}
\usepackage{amsfonts}
\usepackage{amssymb}
\usepackage{times}
\usepackage{color}
\usepackage{graphicx}

\newcommand{\Hom}{\mathrm{Hom}}
\newcommand{\aff}{\mathrm{Trans}}

\newtheorem{lemma}{Lemma}[section]

\newtheorem{theorem}{Theorem}

\newtheorem{proposition}{Proposition}[section]
\newtheorem{definition}{Definition}[section]
\def\Hom{\mbox{Hom}\,}

\def\<{\langle}
\def\>{\rangle}

\def\bA{{\bf A}}
\def\bB{{\bf B}}

\def\rk{{\rm rk}}
\begin{document}

\title{Structure of finite nilspaces and inverse theorems for the Gowers norms in bounded exponent groups.}
\author{{\sc Bal\'azs Szegedy}}

\maketitle

\abstract{A result of the author shows that the behavior of Gowers norms on bounded exponent abelian groups is connected to finite nilspaces. Motivated by this, we investigate the structure of finite nilspaces. As an application we prove inverse theorems for the Gowers norms on bounded exponent abelian groups. It says roughly speaking that if a function on $A$ has non negligible $U_{k+1}$-norm then it correlates with a phase polynomial of degree $k$ when lifted to some abelian group extension of $A$. 
This result is closely related to a conjecture by Tao and Ziegler. In prticular we obtain a new proof for the Tao-Ziegler inverse theorem.}

\tableofcontents

\section{Introduction}

In a recent paper \cite{Sz4} the author proved a general regularity lemma and inverse theorem for Gowers's uniformity norms $U_{k+1}$. The results in \cite{Sz4} connect the theory of certain algebraic structures, called nilspaces \cite{HKr2},\cite{NP}, with the theory of Gowers norms. Note that nilspaces are parallelepiped structures (introduced in a fundamental paper by Host and Kra \cite{HKr2}) in which cubes of every dimension are defined.

The main result in \cite{Sz4} says (very roughly speaking) that every function $f:A\rightarrow\mathbb{C}$ with $|f|\leq 1$ on a compact abelian group can be decomposed as $f=f_s+f_e+f_r$ where $f_e$ is a small error term, $f_s$ is a structured function related to an algebraic morphism $\phi:A\rightarrow N$ where $N$ is a bounded complexity $k$-step nilspace, and $\|f_r\|_{U_{k+1}}$ is very small in terms of the complexity of $N$.

It turns out that if $A$ is chosen from the special family of groups in which the order of every element divides a fixed number $q$ (called $q$ exponent groups) then the nilspace $N$ in the previous decomposition is finite and has exponent $q$ in the sense that it is built up from $k$ abelian groups of exponent $q$ as an iterated bundle.
This motivates us to analyze the structure of finite nilspaces.

Our main method is to produce every finite nilspace $N$ as a factor of a nilspace $M$ which is the direct product of {\it cyclic nilspaces}. (A cyclic nilspace is a cyclic abelian group endowed with a $k$-degree structure for some natural number $k$.) We show that every morphism $\phi:A\rightarrow N$ from an abelian group $A$ to $N$ can be lifted as a morphism $\psi:B\rightarrow M$ where $A$ is a factor group of $B$.

Using these results combined with the results in \cite{Sz4} we get inverse theorems for the Gowers norms in bounded exponent groups. The next definition is from \cite{TZ}

\begin{definition} A phase polynomial $\phi:A\rightarrow\{x|x\in\mathbb{C},|x|=1\}$ of degree $k$ on an abelian group $A$ is a function which trivializes after $k+1$ consecutive application of operators $\Delta_t$ defined by $\Delta_t f(x)=f(x)\overline{f(x+t)}$ and $t\in A$. 
\end{definition}

\begin{theorem}[Inverse theorem for bounded exponent groups]\label{boundexinv}
Let $q$ be a fixed natural number. For every $\epsilon>0$ there is $\delta>0$ such that if $f:A\rightarrow\mathbb{C}$ with $|f|\leq 1$ is a function on the abelian group $A$ of exponent $q$ with $\|f\|_{U_{k+1}}\geq\epsilon$ then there is an extension $B$ of $A$ with the same rank as $A$ and phase polynomial function $\phi:B\rightarrow\mathbb{C}$ of degree $k$ such that $(f,\phi)\geq\delta$.
\end{theorem}

Note that in the above theorem $f$ is interpreted also as a function on $B$ by composing the projection $B\rightarrow A$ with $f$. This makes it possible to take the scalar product $(f,\phi)$.
We can also formulate this inverse theorem in a trivially equivalent but slightly more conventional form if we introduce the notion of a projected phase polynomial.

\begin{definition} A function $f$ is called a {\bf projected phase polynomial} (of degree $k$) on a finite abelian group $A$ if there is an extension $$0\rightarrow C\rightarrow B\stackrel{\tau}{\rightarrow} A\rightarrow 0$$ with $\rk(A)=\rk(B)$ and a phase polynomial $\phi:B\rightarrow\mathbb{C}$ (of degree $k$) such that
$$f(a)=|C|^{-1}\sum_{b\in B,\tau(b)=a}\phi(b).$$
\end{definition}
Using this definition, theorem \ref{boundexinv} says the following.

\medskip

\noindent {\it If $\|f\|_{U_{k+1}}\geq\epsilon$ then $f$ correlates with a projected phase polynomial of degree $k$.}

\medskip
It is clear that projected phase polynomials are "purely structured" functions. This means that correlation with them implies non negligible $U_{k+1}$ norm. To see this, assume that $g$ correlates with a projected phase polynomial. Then on the extended group $B$ the function $g'=\tau\circ g$ correlates with a phase polynomial of degree $k$ and so $\|g'\|_{U_{k+1}}$ is non negligible. On the other hang $\|g'\|_{U_{k+1}}=\|g\|_{U_{k+1}}$.

It will be important that theorem \ref{boundexinv} can be strengthened in two ways. 

\begin{theorem}[Extended inverse theorem]\label{boundext} In theorem \ref{boundexinv} we can also assume that
\begin{enumerate} 
\item $\phi^m=1$ for some $m=q^i$ where $i$ is bounded in terms of $\epsilon$,
\item $\phi=\phi_1\phi_2$ where $\phi_1$ is a degree $k-1$ phase polynomial and $\phi_2$ takes only $q$-th roots of unities.
\end{enumerate} 
\end{theorem}

We will show that theorem \ref{boundext} implies the next theorem by Tao and Ziegler \cite{TZ}.

\begin{theorem}[Tao-Ziegler]\label{TZth}
Let $p$ be a fixed prime and $k\leq p-1$. For every $\epsilon>0$ there is $\delta>0$ such that if $f:A\rightarrow\mathbb{C}$ with $|f|\leq 1$ is a function on the abelian group $A$ of exponent $p$ with $\|f\|_{U_{k+1}}\geq\epsilon$ then there is a phase polynomial function $\phi:A\rightarrow\mathbb{C}$ of degree $k$ and with $\phi^p=1$ such $(f,\phi)\geq\delta$.
\end{theorem}
Note that in \cite{TZ} the authors also prove an inverse theorem for $k>p-1$ in which the function correlates with a phase polynomial of degree bounded in terms of $k$.

\section{Structure of finite nilspaces}

Roughly speaking, a nilspace is a structure in which cubes of every dimension are defined and they behave very similarly as cubes in abelian groups. 
An abstract cube of dimension $n$ is the set $\{0,1\}^n$.
A cube of dimension $n$ in an abelian group $A$ is a function $f:\{0,1\}^n\rightarrow A$ which extends to an affine homomorphism (a homomorphism plus a translation) $f':\mathbb{Z}^n\rightarrow A$. Similarly, a morphism $\psi:\{0,1\}^n\rightarrow\{0,1\}^m$ between abstract cubes is a map which extends to an affine morphism from $\mathbb{Z}^n\rightarrow\mathbb{Z}^m$.

\begin{definition}[Nilspace axioms] A nilspace is a set $N$ and a collection $C^n(N)\subseteq N^{\{0,1\}^n}$ of functions (or cubes) of the form $f:\{0,1\}^n\rightarrow N$ such that the following axioms hold.
\begin{enumerate}
\item {\bf(Composition)} If $\psi:\{0,1\}^n\rightarrow\{0,1\}^m$ is a cube morphism and $f:\{0,1\}^m\rightarrow N$ is in $C^m(N)$ then the composition $\psi\circ f$ is in $C^n(N)$.
\item {\bf(Ergodictiry)} $C^1(N)=N^{\{0,1\}}$.
\item {\bf(Gluing)} If a map $f:\{0,1\}^n\setminus\{1^n\}\rightarrow N$ is in $C^{n-1}(N)$ restricted to each $n-1$ dimensional face containing $0^n$ then $f$ extends to the full cube as a map in $C^n(N)$.
\end{enumerate} 
\end{definition}

\bigskip

If $N$ is a nilspace and in the third axiom the extension is unique for $n=k+1$ then we say that $N$ is a $k$-step nilspace.
If a space $N$ satisfies the first axiom (but the last two are not required) then we say that $N$ is a {\bf cubespace}. 
A function $f:N_1\rightarrow N_2$ between two cubespaces is called a {\bf morphism} if $\phi\circ f$ is in $C^n(N_2)$ for every $n$ and function $\phi\in C^n(N_1)$.
The set of morphisms between $N_1$ and $N_2$ is denoted by $\Hom(N_1,N_2)$. With this notation $C^n(N)=\Hom(\{0,1\}^n,N)$.
If $N$ is a nilspace then every morphism $f:\{0,1\}^n\rightarrow\{0,1\}^m$ induces a map $\hat{f}:C^m(N)\rightarrow C^n(N)$ by simply composing $f$ with maps in $C^m(N)$. 

Every abelian group $A$ has a natural nilspace structure in which cubes are maps $f:\{0,1\}^n\rightarrow A$ which extend to affine homomorphisms $f':\mathbb{Z}^n\rightarrow A$. We refer to this as the linear structure on $A$. For every natural number $k$ we can define a nilspace structure $\mathcal{D}_k(A)$ on $A$ that we call the $k$-degree structure.
A function $f:\{0,1\}^n\rightarrow A$ is in $C^n(\mathcal{D}_k(A))$ if for every cube morphism $\phi:\{0,1\}^{k+1}\rightarrow\{0,1\}^n$ we have that
$$\sum_{v\in\{0,1\}^{k+1}}f(\phi(v))(-1)^{h(v)}=0$$
where $h(v)=\sum_{i=1}^{k+1}v_i$.
We will use the next elementary lemma in this paper.

\begin{lemma}\label{derivmorph} Let $\phi:\mathcal{D}_i(A)\rightarrow\mathcal{D}_j(B)$ be a morphism. Then if $i>j$ then $\phi$ is constant.
If $1\leq i\leq j$ then $\phi$ is also a morphism from $\mathcal{D}_1(A)$ to $\mathcal{D}_{j-i+1}(B)$.
\end{lemma}
This follows directly by applying the operators $\partial_x$ (defined in \cite{NP}) to the cub structures on $\mathcal{D}_i(A)$ and $\mathcal{D}_j(B)$.

It was shown in \cite{NP} that higher degree abelian groups are building blocks of every $k$-step nilspace.
To state the precise statement we will need the following formalism.
Let $A$ be an abelian group and $X$ be an arbitrary set. An $A$ bundle over $X$ is a set $B$ together with a free action of $A$ such that the orbits of $A$ are parametrized by the elements of $X$. This means that there is a projection map $\pi:B\rightarrow X$ such that every fibre is an $A$-orbit. 
The action of $a\in A$ on $x\in B$ is denoted by $x+a$. Note that if $x,y\in B$ are in the same $A$ orbit then it make sense to talk about the difference $x-y$ which is the unique element $a\in A$ with $y+a=x$. In other words the $A$ orbits can be regarded as affine copies of $A$.

A $k$-fold abelian bundle $X_k$ is a structure which is obtained from a one element set $X_0$ in $k$-steps in a way that in the $i$-th step we produce $X_i$ as an $A_i$ bundle over $X_{i-1}$.
The groups $A_i$ are the structure groups of the $k$-fold bundle. We call the spaces $X_i$ the $i$-th factors. 

\begin{definition}\label{bundef} Let $X_k$ be a $k$-fold abelian bundle with factors $\{X_i\}_{i=1}^k$ and structure groups $\{A_i\}_{i=1}^k$. Let $\pi_i$ denote the projection of $X_k$ to $X_i$.
Assume that $X_k$ admits a cubespace structure with cube sets $\{C^n(X_k)\}_{n=1}^\infty$.
We say that $X_k$ is a {\bf $k$-degree bundle} if it satisfies the following conditions
\begin{enumerate}
\item $X_{k-1}$ is a $k-1$ degree bundle.
\item Every function $f\in C^n(X_{k-1})$ can be lifted to $f'\in C^n(X_k)$ with $f'\circ\pi_{k-1}=f$.
\item If $f\in C^n(X_k)$ then the fibre of $\pi_{k-1}:C^n(X_k)\rightarrow C^n(X_{k-1})$ containing $f$ is  $$\{f+g|g\in C^n(\mathcal{D}_k(A_k))\}.$$ 
\end{enumerate} 
\end{definition}

The next theorem form \cite{NP} says that $k$-degree bundles are the same as $k$-step nilspaces.

\begin{theorem} Every $k$-degree bundle is a $k$-step nilspace and every $k$-step nilspace arises as a $k$-degree bundle. 
\end{theorem}

We say that a $k$-step nilspace is of exponent $q$ if all the structure groups $A_i$ are of exponent $q$.

\begin{definition} Let $x,y$ be two elements in the nilspace $N$. We say that $x\sim_i y$ if the map $c:\{0,1\}^{i+1}\rightarrow N$ with $c(0^{i+1})=x$ and $c(v)=y$ if $v\neq 0^{i+1}$ is an element in $C^{i+1}(N)$.
It was proved in \cite{NP} that $\sim_i$ is an equivalence relation and the classes form a factor of $N$ denoted by $\mathcal{F}_i(N)$. The factor $\mathcal{F}_i$ coincides with the factor $X_i$ in definition \ref{bundef}.
\end{definition}

We introduce special morphisms between $k$-step nilspaces which have very strong surjectivity properties and behave consistently with respect to the equivalence classes $\sim_i$.

\begin{definition} Let $N$ and $M$ be $k$-step nilspaces. A morphism $\phi:N\rightarrow M$ is called a factor map if 
for every $1\leq i\leq k$ the image of every $\sim_i$ class in $N$ is a $\sim_i$ class in $M$.
\end{definition}

Notice that every morphism maps a $\sim_i$ class into a $\sim_i$ class. In the above definition we require the surjectivity of these local maps.
These type of maps were also investigated in \cite{NP}.
If $N$ and $M$ are compact nilspaces then factor maps are measure preserving which is useful.
In this paper we will use this notion to obtain finite nilspaces from free nilspaces as factors.

\subsection{Extensions}

\begin{definition}\label{kdegext} Let $N$ be an arbitrary nilspace. A degree $k$-extension of $N$ is an abelian bundle $M$ over $N$ which is a cube space with the following properties.
\begin{enumerate}
\item For every $n\in\mathbb{N}$ and $c\in C^n(N)$ there is $c'\in C^n(N)$ such that $\pi(c')=c$,
\item If $c_1\in C^n(M)$ and $c_2:\{0,1\}^n\rightarrow M$ with $\pi(c_1)=\pi(c_2)$ then $c_2\in C^n(M)$ if and only if $c_1-c_2\in C^n(\mathcal{D}_k(A))$.
\end{enumerate}
The map $\pi$ is the projection from $M$ to $N$.
The extension $M$ is called a split extension if there is a cube preserving morphism $m:N\rightarrow M$ such that $m\circ\pi$ is the identity map of $N$.
\end{definition}

\begin{lemma} Let $N$ be a $k$-step nilspace. If $M$ is a degree $i$ extension of $N$ by an abelian group $A$ then $\mathcal{F}_i(M)$ is a degree $i$ extension of $\mathcal{F}_i(N)$ by $A$. Furthermore the projection $\tau:\mathcal{F}_i(M)\rightarrow\mathcal{F}_i(N)$ is the composition of the projection $\pi:M\rightarrow N$ by the projection $\pi_i:N\rightarrow\mathcal{F}_i(N)$. 
\end{lemma}

\begin{proof} First of all observe that $\mathcal{F}_i(M)$ is some $i$-degree extension of $\mathcal{F}_i(N)$. All we need to show is that the structure group of this extension is $A$. In other words $A$ does not collapse when we look at the situation modulo $\sim_i$. To see this we show that if $x,y$ are in the same fibre of $\pi$ and $x\neq y$ then $x$ is different from $y$ mod $\sim_i$.
Let $F$ be the fibre containing $x$ and $y$. Since the structure of $M$ restricted to $F$ is $\mathcal{D}_i(A)$ and $\sim_i$ separates every element in $\mathcal{D}_i(A)$ the proof is complete. 
\end{proof}

\begin{lemma}\label{subdirect} Let $N$ be a $k$-step nilspace. If $M$ is a degree $i$ extension of $N$ by an abelian group $A$ (with projection $\pi$) then $M$ is a sub direct product of $N$ with a degree $i$ extension $K$ of $\mathcal{F}_i(N)$ by $A$. The sub direct product is the set of pairs $(a,b)$ such that $\pi_i(a)=\tau(b)$ such that $\pi_i:N\rightarrow\mathcal{F}_i(N)$ and $\tau:K\rightarrow\mathcal{F}_i(N)$ are the projection maps.
\end{lemma}

\begin{proof} Let $K=\mathcal{F}_i(M)$ and $\tau:M\rightarrow\mathcal{F}_i(M)$ be the projection. Let $\phi=\pi\times\tau$ be the morphism from $M$ to $N\times K$. According to the previous lemma the image of $\phi$ gives an isomorphism between $M$ and the subdirect product of $N$ and $K$ defined in the lemma.
\end{proof}

\subsection{Translation groups}

For an arbitrary subset $F$ in $\{0,1\}^n$ and map $\alpha:N\rightarrow N$ we define the map $\alpha^F$ from $C^n(N)$ to $N^{\{0,1\}^n}$ such that $\alpha^F(c)(v)=\alpha(c(v))$ if $v\in F$ and $\alpha^F(c)(v)=c(v)$ if $v\notin F$.

\begin{definition}\label{transdef} Let $N$ be a nilspace. A map $\alpha:N\rightarrow N$ is called a translation of hight $i$
if for every natural number $n\geq i$, $n-i$ dimensional face $F\subseteq\{0,1\}^n$ and $c\in C^n(N)$ the map $\alpha^F(c)$ is in $C^n(N)$. We denote the set of hight $i$ translations by $\aff_i(N)$. We will use the short hand notation $\aff(N)$ for $\aff_1(N)$.
\end{definition}

It is not hard to see that if $N$ is a $k$-step nilspace then $\aff(N)$ is a $k$-nilpotent group and $\{\aff_i(N)\}_{i=1}^k$ is a central series in $\aff(N)$.
In this chapter we are interested in the following question.

\medskip

\noindent{\it Let $N$ be an $k$-degree extension of a $k-1$ step nilspace $M$ and let $\alpha\in\aff_i(M)$. Under what circumstances can we lift $\alpha$ to an element $\alpha'\in\aff_i(N)$ such that $\pi(\alpha'(n))=\alpha(\pi(n))$ for every $n\in N$ ($\pi:N\rightarrow M$ is the projection.)?} 

\medskip

We will need the definition of the arrow spaces.
Let $f_1,f_2:\{0,1\}^n\rightarrow N$ be two maps. We denote by $(f_1,f_2)_i$ the map $g:\{0,1\}^{n+i}\rightarrow N$ such that $g(v,w)=f_1(v)$ if $w\in\{0,1\}^i\setminus\{1^i\}$ and $g(v,w)=f_2(v)$ if $w=1^i$. If $f:\{0,1\}^n\rightarrow N\times N$ is a single map with components $f_1,f_2$ then we denote by $(f)_i$ the map $(f_1,f_2)_i$.
A map $f:\{0,1\}^n\rightarrow N\times N$ is a cube in the $i$-th arrow space if $(f)_i$ is a cube in $N$.

Let $\mathcal{T}=\mathcal{T}(\alpha,N,i)$ be the set of pairs $(x,y)\in N^2$ where $\alpha(\pi_{k-1}(x))=\pi_{k-1}(y)$. We interpret $\mathcal{T}$ as a subset of the $i$-th arrow space over $N$.
It is easy to see that if $k\geq i+1$ then $\mathcal{T}$ is an ergodic nilspace with the inherited cubic structure.
We define $\mathcal{T}^*$ as $\mathcal{F}_{k-1}(\mathcal{T})$.
We have \cite{NP} that $\mathcal{T}^*$ is a degree $k-i$ extension of $\mathcal{F}_{k-1}(N)$ by $A_k$.
The next theorem was proved in \cite{NP}

\begin{proposition}\label{transext} Let $N$ be a $k$-step nilspace and $\alpha\in\aff_i(\mathcal{F}_{k-1}(N))$. If $\mathcal{T}^*=\mathcal{T}^*(\alpha,N,i)$ is a split extension then $\alpha$ lifts to an element $\beta\in\aff_i(N)$.
\end{proposition}

\subsection{Free nilspaces and their extensions}

\begin{definition} Let $a_1,a_2,\dots,a_k$ be a sequence of natural numbers. We denote by $F(a_1,a_2,\dots,a_k)$ the nilspace 
$$\prod_{i=1}^k\mathcal{D}_i(\mathbb{Z}^{a_i}).$$
We say that $F(a_1,a_2,\dots,a_k)$ is the free nilspace of rank $(a_1,a_2,\dots,a_k)$.
\end{definition}

\begin{lemma}\label{cubcub} Let $N$ be a nilspace. If $c\in C^n(\mathcal{D}_i(\mathbb{Z}))$, $f\in C^n(N)$ and $\alpha\in\aff_i(N)$.
Then $f^c$ defined by $v\rightarrow \alpha^{c(v)}(f(v))$ (where $v\in\{0,1\}^n$) is in $C^n(N)$. 
\end{lemma}

\begin{proof} If $c$ has the special structure that it takes a value $a\in\mathbb{Z}$ on a face of co-dimension $i$ and takes $0$ on the rest of $\{0,1\}^n$ then the statement follows directly from the definition of $\aff_i(N)$.
Every other cube in $C^n(\mathcal{D}_i(\mathbb{Z}))$ can be generated by such simple cubes, so we get the general case by iterating the special case.
\end{proof}

The main result of this chapter is the following.

\begin{theorem}\label{split} Let $M$ be a degree $d$ extension of the free nilspace $F=F(a_1,a_2,\dots,a_k)$ by an abelian group $A$. Then $M$ is a split extension.
\end{theorem}

\begin{proof} We prove the statement by induction on $d$. If $d=1$ then using lemma \ref{subdirect} we get that the space $M$ is the direct product of $F(0,a_2,a_3,\dots,a_k)$ with an abelian group extension of $\mathbb{Z}^{a_1}$ by $A$. Since such an extension splits the case $d=1$ is done.

Assume that the statement holds for $d-1$ and $d\geq 2$.
Using lemma \ref{subdirect} we get that $M$ is the direct product of 
$\prod_{i=d+1}^k\mathcal{D}_i(\mathbb{Z}^{a_i})$ with a degree $d$ extension $M_d$ of $F_d=F(a_1,a_2,\dots,a_d)$ by $A$.
We also have that the $d$-th structure group $B$ of $M_d$ is a $d$-degree extension of $\mathbb{Z}^{a_d}$ by $A$.
It follows that $B$ is the $d$-degree structure on an abelian group extension of $\mathbb{Z}^{a_d}$ by $A$. 
Such an extension splits so it remains to show that the degree $d$ extension $M_d$ of $F_{d-1}=F(a_1,a_2,\dots,a_{d-1})$ by $B$ splits. In other words we reduced the problem to the case when $k=d-1$. By abusing the notation let us assume that $k=d-1$, $B=A$ and $F=F_{d-1}$.

We use that $\mathcal{D}_i(\mathbb{Z}^{a_i})$ is embedded into $F$ by setting all the other coordinates to $0$.
Let $S_i$ be a free generating system in $\mathbb{Z}^{a_i}$ embedded into $F$ this way (for every $1\leq i\leq d-1$).
Every element in $g\in S_i$ acts on $F$ by $x\mapsto x+g$ using the abelian group addition in $\prod_{i=1}^{d-1}\mathbb{Z}^{a_i}$.
Let us denote this action by $\alpha(g)$.
It is clear that $\alpha(g)\in\aff_i(F)$.
We claim that $\alpha(g)$ can be lifted to $M$.
Let $\mathcal{T}=\mathcal{T}(\alpha(g),M,i)$. We have by proposition \ref{transext} that $\mathcal{T}$ is an $d-i$ degree extension of $F_{d-1}$ by $A$.
Using our induction this extension splits and so there is a lift $\alpha'(g)$ of $\alpha(g)$ to $\aff_i(M)$.

The last step of the proof is to create a complement of $A$ in $M$ using the group elements $\alpha'(g)$ where $g\in S_i$.
Assume that $S_i=\{g_{i,1},g_{i,2},\dots,g_{i,a_i}\}$. We represent every element $x$ in $F$ in a unique way as
$$x=\sum_{i=1}^{d-1}\sum_{j=1}^{a_i}\lambda_{i,j}g_{i,j}.$$
Let $m\in M$ be a fixed element in the fibre of $0\in F$ in $M$. We define the map $h:F\rightarrow M$ such that $h(x)$ is the image of $m$ under the transformation
\begin{equation}\label{noncommpr}
\prod_{i=1}^{d-1}\prod_{j=1}^{a_i}\alpha'(g_{i,j})^{\lambda_{i,j}}.
\end{equation}

Note that the order in the above product is important since we multiply non commuting transformations. It remains to show that $h$ is cube preserving and $h(F)$ is a diagonal embedding of $F$ into $M$.
Lemma \ref{cubcub} shows that if a product of the form (\ref{noncommpr}) gets extended by one more term then it remains cube preserving. By induction on the length of (\ref{noncommpr}) we get that $h$ is a morphism.
It is clear from its definition that $h$ creates a diagonal embedding. 
\end{proof}

\subsection{Finite nilspaces as factors of free nilspaces}

In this chapter we establish finite nilspaces as factors of free nilspaces.
Note that in this paper free nilspaces are defined to have finite rank.

\begin{theorem}\label{freefact} For every finite $k$-step nilspace $N$ there is a factor map $h:F\rightarrow N$ from a $k$-step free nilspace $F$ with the following property. If $\phi:\mathbb{Z}^n\rightarrow N$ is a morphism then there is a morphism $\phi':\mathbb{Z}^n\rightarrow F$ such that $\phi'\circ h=\phi$.
\end{theorem}

\begin{proof} We proceed by induction on $k$. If $k=1$ the $N$ is an (affine) abelian group and then the result is classical.
Assume that $k\geq 2$ and assume that the statement holds for $k-1$.
Let $N$ be a fixed $k$-step nilspace.
We can regard $N$ as a $k$-degree extension of a $k-1$ step nilspace $M$ by an abelian group $A$. Let $\pi:N\rightarrow M$ be the projection. 
We use the induction hypothesis for $M$ and construct a free nilspace $F_{k-1}=F(a_1,a_2,\dots,a_{k-1})$ and factor map $h':F_{k-1}\rightarrow M$ satisfying the requirement of the lemma.
Let $Q$ be the subdirect product of $F_{k-1}$ and $N$ in the following way. The set $Q$ consists of the pairs $(a,b)$ such that $a\in F_{k-1}$,$b\in N$ and $h'(a)=\pi(b)$. It is clear that $Q$ as a subset of the nilspace $F_{k-1}\times N$ is a nilspace which is a $k$ degree extension of $F_{k-1}$ by $A$. It follows form theorem \ref{split} that $Q$ is a split extension and so $Q\simeq F_{k-1}\times\mathcal{D}_k(A)$. Let $\beta:\mathbb{Z}^r\rightarrow A$ be a surjective homomorphism between abelian groups for some natural number $r$. Then $\beta$ is also a morphism from $\mathcal{D}_k(\mathbb{Z}^r)$ to $\mathcal{D}_k(A)$.
Let $F=F_{k-1}\times\mathcal{D}_k(\mathbb{Z}^r)$ and $\gamma:F\rightarrow Q$ be the identity map on $F_{k-1}$ times $\beta$.
Let $h:F\rightarrow N$ be the composition of $\gamma$ with the projection from $Q$ to the second coordinate.
It is clear that $h$ is a factor map.
We claim that $h$ has the desired lifting property.

Let $\phi:\mathbb{Z}^n\rightarrow N$ be a morphism. First we lift $\phi$ to $Q$. According to induction we can lift $\phi\circ\pi$ to a morphism $\phi_2:\mathbb{Z}^n\rightarrow F_{k-1}$. Let $\phi_3=\phi_2\times\phi$. It is clear that $\phi_3$ maps $\mathbb{Z}^n$ to $Q$ and it lifts $\phi$. Now we have to further lift $\phi_3$ from $Q$ to $F$. Since $Q=F_{k-1}\times\mathcal{D}_k(A)$ we can write $\phi_3$ as $\phi_4\times\phi_5$ where $\phi_4:\mathbb{Z}^n\rightarrow F_{k-1}$ and $\phi_5:\mathbb{Z}^n\rightarrow\mathcal{D}_k(A)$. It remains to show that $\phi_5$ can be lifted to $\mathcal{D}_k(A)$.
The map $\phi_5$ is a degree $k$ polynomial map from $\mathbb{Z}^n$ to $A$.
An easy lemma in \cite{Sz4} show exactly that such a map can be lifted.
\end{proof}

\subsection{Periodicity of morphisms}

\begin{lemma}\label{specper} Let $A$ be a finite abelian group of order $n$. Then for every $k\in\mathbb{N}$ there is a natural number $\alpha$ depending on $n$ and $k$ such that any morphism $\phi:\mathcal{D}_i(\mathbb{Z})\rightarrow\mathcal{D}_k(A)$ is $n^\alpha$ periodic.
\end{lemma}

\begin{proof} We use lemma \ref{derivmorph}.
If $i>k$ then $\phi$ is constant.
If $i\leq k$ then $\phi$ is a morphism from $\mathcal{D}_1(\mathbb{Z})$ to $\mathcal{D}_{k-i+1}(A)$.
It was proved in \cite{Sz4} (but it is also very easy to see) that such a function is of the form $m\rightarrow \sum_{j=0}^{k-i+1}x_j{{m}\choose{j}}$ where $x_i\in A$ for every $i$. 
Such functions are $n^{k-i+1}$ periodic.
\end{proof}

\begin{lemma}\label{per} Let $N$ be a finite $k$-step nilspace of size $|N|=n$. then there is a natural number $\alpha\in\mathbb{N}$ such that any morphism $\phi:\mathcal{D}_i(\mathbb{Z})\rightarrow N$ is $n^\alpha$ periodic.
\end{lemma}

\begin{proof} We prove the statement by induction on $k$. If $k=1$ then the statement (by lemma \ref{specper}) is trivial since in this case $\phi$ is an affine morphism of $\mathcal{D}_i(\mathbb{Z})$ into an abelian group. Assume that $k\geq 2$ and the statement if true for $k-1$. Assume that the structure groups of $N$ are $A_1,A_2,\dots,A_k$. Let $n_2=|\mathcal{F}_{k-1}(N)|=\prod_{i=1}^{k-1}|A_i|$. We have that $n_2$ divides $n$. We also have by induction that any morphism $\phi:\mathbb{Z}\rightarrow N$ composed with the projection map $\pi:N\rightarrow\mathcal{F}_{k-1}(N)$ is $n_2^{\alpha_2}$ periodic for some natural number $\alpha_2$. This means that for any natural number $m$ the sequence $j\mapsto\phi(m+n_2^{\alpha_2}j)$ is a morphism of $\mathcal{D}_i(\mathbb{Z})$ into a single fibre of the projection $\pi$. 

The cube structure of a fibre of $\pi$ is $\mathcal{D}_k(A)$. Lemma \ref{specper} finishes the proof.
\end{proof}

\begin{theorem}\label{moding} Let $\phi:F\rightarrow N$ be a morphism from a free group $F=F(a_1,a_2,\dots,a_k)$ to a finite nilspace $F$ with $|F|=n$. Then there is a natural number $\alpha$ such that $\phi$ factors through the map $\psi:F\rightarrow F/n^\alpha$ where $F/n^\alpha$ is the space $\prod_{i=1}^k\mathcal{D}_i((\mathbb{Z}/(n^\alpha))^{a_i})$ and $\psi$ is the map which takes every coordinate mod $n^\alpha$.
\end{theorem}

\begin{proof} It follows from lemma \ref{per} that if $g$ is an element in $F$ which has only one nonzero coordinate then the map $j\rightarrow\phi(z+g(jn^\alpha))$ on $\mathbb{Z}$ is constant for every $z\in F$ using addition in the abelian group $\prod_{i=1}^k\mathbb{Z}^{a_i}$. This periodicity using all the generators of each component $\mathbb{Z}^{a_i}$ implies the statement.
\end{proof}

\subsection{Lifting morphisms}

\begin{definition} A hight $i$ extension of an abelian group $A$ of rank $r$ and exponent $e$ is an abelian group $B$ which is an extension of $A$ with rank $r$ and exponent dividing $e^i$.
\end{definition}

\begin{definition} We denote by $F_n(a_1,a_2,\dots,a_k)=\prod_{i=1}^k\mathcal{D}_i(\mathbb{Z}_n^{a_k})$ and we call it the modulo $n$ free nilspace. 
\end{definition}

\begin{lemma}\label{invprep} Let $N$ be a finite nilspace of exponent $e$, $A$ be an abelian group of exponent $e$ and $\phi:A\rightarrow N$ be a morphism. Then there is a morphism $\psi:B\rightarrow F$ from some hight $i$ extension $B$ of $A$ to a modulo $e^{\alpha}$ free nilspace $F=F_{e^\alpha}(a_1,a_2,\dots,a_k)$ such that there is a factor map $\beta:F\rightarrow N$ with $\psi\circ\beta=\pi\circ\phi$ where $\pi:B\rightarrow A$ is the projection map. The value $\alpha$ and the number $i$ depends only on the structure of $N$.
\end{lemma}

\begin{proof} Assume that $A$ is of rank $r$. This means that we can write $A$ as a factor group of $\mathbb{Z}^r$. Theorem \ref{freefact} and theorem \ref{moding} imply that the statement is true if $B$ is replaced by $\mathbb{Z}^d$. It remains to show that a morphisms $\psi':\mathbb{Z}^d\rightarrow F_{e^\alpha}(a_1,a_2,\dots,a_k)$ factors through a bounded hight extension of $A$. This follows form lemma \ref{per} since $\psi'$ has to be periodic in each coordinate with a bounded power of $e$.
\end{proof}

We get a further strengthening of the previous lemma from using the fact that $\beta$ is a factor map. Since the $\sim_{k-1}$ classes of $F$ are mapped surejectively to the $\sim_{k-1}$ classes in $N$ (and the $\sim_{k-1}$ classes are copies of $\mathcal{D}_k(A_k)$) the map $\beta$ factors through the map 
\begin{equation}\label{redukalt}
F_{e^\alpha}(a_1,a_2,\dots,a_k)\rightarrow F_{e^\alpha}(a_1,a_2,\dots,a_{k-1})\times\mathcal{D}_{k}(A_k)=F'
\end{equation}
where $A_k$ is the $k$-th structure group of $N$.
We formulate it as a separate lemma.

\begin{lemma}\label{invprep2} In Lemma \ref{invprep} the nilspace $F$ can be replaced by $F'$ in (\ref{redukalt}).
\end{lemma}

\section{Applications to the Gowers norms}

Let $\mathfrak{A}_e$ denote the family of finite abelian groups of exponent $e$.
A $\mathfrak{A}_e$ nilspace is a nilspace whose structure groups are all in $\mathfrak{A}_e$.
In particular such nilspaces are all finite.
A $\mathfrak{A}_e$ nilspace polynomial on $A\in\mathfrak{A}_e$ is a composition of a morphism $\phi:A\rightarrow N$ with $g:N\rightarrow\mathbb{C}$ where $N$ is a $\mathfrak{A}_e$ nilspace and $|g|\leq 1$.

The next regularity lemma was proved in \cite{Sz4}.
\begin{theorem}[Regularization in $\mathfrak{A}_e$.]\label{reglem} Let $k$ be a fixed number and $F:\mathbb{R}^+\times\mathbb{N}\rightarrow\mathbb{R^+}$ be an arbitrary function. Then for every $\epsilon>0$ there is a number $n=n(\epsilon,F)$ such that for every measurable function $f:A\rightarrow\mathbb{C}$ on $A\in\mathfrak{A}_e$ with $|f|\leq 1$ there is a decomposition $f=f_s+f_e+f_r$ and number $m\leq n$ such that the following conditions hold.
\begin{enumerate}
\item $f_s$ is a degree $k$, complexity $m$ and $F(\epsilon,m)$-balanced $\mathfrak{A}_e$-nilspace-polynomial,
\item $\|f_e\|_1\leq\epsilon$,
\item $\|f_r\|_{U_{k+1}}\leq F(\epsilon,m)$~,~$|f_r|\leq 1$ and $|(f_r,f_s+f_e)|\leq F(\epsilon,m)$.
\end{enumerate}
\end{theorem}
The notion of $b$-balanced is not crucial in this paper but it expresses a strong measure preserving property.
According to this theorem, if we want to get an inverse theorem in $\mathfrak{A}_e$ all we need to do is to write $\mathfrak{A}_e$ nilspace polynomials as linear combinations in a suitable basis $\mathcal{B}$. If we can have such a linear combination with boundedly many elements then we get correlation with a function from $\mathcal{B}$.

Let $\phi:A\rightarrow N$ be a morphism from $A\in\mathfrak{A}_e$ to a $\mathfrak{A}_e$-nilspace $N$. Assume that $N$ is a $k$-step nilspace and has size at most $m$.
We use lemma \ref{invprep2} for $\phi$. We obtain a morphism $\psi:B\rightarrow F'$ where $F'$ has the form (\ref{redukalt}).
The group $B\in\mathfrak{A}_{e^i}$ is a bounded hight extension of $A$. Let $\phi'$ denote the composition of the projection $B\rightarrow A$ with $\phi$.
We have that $\phi'$ factors through $\psi$.
This means that nilspace polynomials using $\phi'$ can be also obtained as nilspace polynomials using $\psi$.
Let $\mathcal{B}$ be the set of functions on $B$ which are obtained by composing a linear character $\chi$ of 
$\prod_{i=1}^{k-1}\mathbb{Z}_{e^\alpha}^{a_i}\times A_k$ with $\psi$. Let $\chi=\prod_{i=1}^k\chi_i$ where $\chi_i$ is a character of $\mathbb{Z}_{e^\alpha}^{a_i}$ is $1\leq i\leq k-1$ and $\chi_k$ is a character of $A_k$.
We have that $\chi(\psi(x))=\prod_{i=1}^k\chi_i(\psi_i(x))$ where $\psi_i$ is a morphism from $B$ to $\mathcal{D}_i(\mathbb{Z}_{e^\alpha}^{a_i})$. In other words $\psi_i$ is a polynomial map of degree $i$ and $x\mapsto\chi_i(\psi_i(x))$ is a phase polynomial of degree $i$. Furthermore since $A_k$ has exponent $e$ we have that $\chi_k(\psi_k)$ takes only $e$-th roots of unities.

This completes the proof of both theorem \ref{boundexinv} and theorem \ref{boundext}

\subsection{On the Tao-Ziegler theorem}

Let $p$ be a fixed prime number. We apply theorem \ref{boundext} inductively on $k$.
Assume that the inverse theorem is established for degree $k-1$ and $k\leq p-1$.
Let $f$ be a function on $A=\mathbb{Z}_p^n$ such that $|f|\leq 1$ and $\|f\|_{U_{k+1}}\geq\epsilon$.
We have by theorem \ref{boundext} that there is a group extension $B\rightarrow A$ (where $B$ has rank $n$) and a phase polynomial $\phi$ on $B$ of degree $k$ such that $(\tau\circ f,\phi)>\delta$.
Furthermore we have that $\phi=\phi_1\phi_2$ where $\phi_1$ is of degree $k-1$ and $\phi_2$ takes only $p$-th roots of unities. 

We claim that $\phi_2=\tau\circ\phi_3$ where $\phi_3$ is a phase polynomial on $A$.
Since $B$ has rank $n$ we can write $B$ as a factor group of $\mathbb{Z}^n$ with homomorphism $\tau':\mathbb{Z}^n\rightarrow B$.
Let us lift $\phi_2$ to a phase polynomial $\phi_4$ on $\mathbb{Z}^n$ by composing it with $\tau'$. 
Since $\phi_4^p=1$ we can obtain $\phi_4$ from a polynomial map $l:\mathbb{Z}^n\rightarrow\mathbb{Z}_p$ of degree $k$.
Such polynomial maps (see \cite{Sz4}) are linear combinations of functions $(x_1,x_2\dots,x_n)\rightarrow \prod_{i=1}^n{{x_i}\choose{r_i}}$ where $\sum r_i=k$.  
This show that the value of any such map depends only of the residue classes of $x_i$ modulo $p$. 
The claim is proved.

Let $\gamma$ be the projection of $\phi_1\phi_2$ to $A$. Since $\phi_2$ factors through $\tau$ we have that $\gamma=\gamma_2\phi_3$ where $\gamma_2$ is the projection of $\phi_1$.
To finish the proof we need to show that $\gamma_2$ can be well approximated by a bounded linear combination of phase polynomials on $A$ of degree $k-1$.
Intuitively this follows from our induction hypothesis and the fact that $\gamma_2$ is a "purely structured" degree $k-1$ function. In the rest of the proof we show a lemma which makes this intuition precise.

\begin{lemma} Let $k$ be a value such that theorem \ref{TZth} holds. Then for every $\epsilon_1,\epsilon_2>0$ there is a number $m$ such that if $\|f\|_{U_{k+1}}\geq\epsilon_1$ holds for a projected phase polynomial $f$ on $A$ then 
$$\|f-\sum_{i=1}^m\lambda_i\beta_i\|_2\leq\epsilon_2$$
holds for some linear combination $\sum\lambda_i\beta_i$ of degree $k$ phase polynomials $\beta_i$ with $|\lambda_i|\leq 1$.
\end{lemma}

\begin{proof} We proceed by contradiction. Let $\epsilon_1,\epsilon_2$ be such that there is a sequence of groups $A_i=\mathbb{Z}_p^{n_i}$ with extensions $B_i$ and phase polynomials $\phi_i$ on $B_i$ such that the projection $f_i$ of $\phi_i$ has $U_{k+1}$ norm at least $\epsilon_1$, but there is no required linear combination with $i$ elements.

Let $\bA$ (resp. $\bB$) be the ultra product of the sequence $\{A_i\}_{i=1}^\infty$ (resp. $\{B_i\}_{i=1}^\infty$). Let $\phi$ be the ultra limit of the sequence $\{\phi_i\}_{i=1}^\infty$ and $f$ be the ultra limit of $\{f_i\}_{i=1}^\infty$. The theory developed in \cite{Sz1} says that there is a maximal $\sigma$-algebra $\mathcal{F}_k(\bB)$ on $\bB$ such that $U_{k+1}$ is a norm on $L^\infty(\mathcal{F}_k(\bB))$.

The next step is to show that the projection of $\phi$ to $\bA$ (which is equal to $f$) is measurable in $\mathcal{F}_k(\bA)$.
We can think about the projection to $\bA$ as the projection to a $\sigma$ algebra generated by the factor map $\bB\rightarrow\bA$. Such $\sigma$-algebras are called coset $\sigma$-algebras in \cite{Sz1} since measurable sets are unions of cosets of the kernel of the morphism from $\bB$ to $\bA$.
It was proved in $\cite{Sz1}$ that the projection of a function measurable in a shift invariant $\sigma$-algebra to a coset $\sigma$ algebra is measurable in the original $\sigma$ algebra. We obtain that a function in $\mathcal{F}_k(\bB)$ projected to $\bA$ is measurable in $\mathcal{F}_k(\bA)$. In particular $f$ is measurable in $\mathcal{F}_k(\bA)$.

Next we observe that $L^2(\mathcal{F}_k(\bA))$ is generated by ultra limits of phase polynomials. Assume that it is not true. Then there is a nonzero function $g$ (with $|g|\leq 1$) measurable in $\mathcal{F}_k(\bA)$ which is orthogonal to the space generated by the ultra limits of phase polynomials. Since $U_{k+1}$ is a norm on $L^\infty(\mathcal{F}_k(\bA))$ we have that $\|g\|_{U_k}=g>0$. Then we choose a sequence of functions $g_i$ on $a_i$ whose ultra limit is $g$. This sequence would contradict the assumption that theorem \ref{TZth} holds for $k$.

We obtain that $f=\sum_{i=1}^\infty w_i\lambda_i$ where $w_i$ are ultra limits of phase polynomials $\{w_i^j\}_{j=1}^\infty$ of degree $k$.
By repeating terms $w_i$ many times we can assume the each lambda has absolute value at most $1$. Then there is $m$ such that 
$\|f-\sum_{i=1}^m\lambda_iw_i\|_2\leq\epsilon_2/2$. This gives a contradiction since 
$\|f_i-\sum_{r=1}^m\lambda_rw_r^i\|_2<\epsilon_2$ holds for infinitely many indices $i$.

\end{proof}

\end{document}